\documentclass[10pt,a4paper,twoside]{article}
\RequirePackage{filecontents}
\usepackage[utf8]{inputenc}
\usepackage[UKenglish]{babel}
\usepackage[T1]{fontenc}
\usepackage{longtable}
\usepackage{caption}
\usepackage{multirow}
\usepackage{fancyhdr}
\usepackage{ulem}
\usepackage[left=3cm,right=3cm,top=4cm,bottom=4cm]{geometry}
\usepackage{etex}
\usepackage[etex=true,export]{adjustbox}
\usepackage{float}
\usepackage{xkeyval}
\usepackage{mathtools}
\usepackage{amsmath}
\usepackage{amsthm}
\usepackage{amsfonts}
\usepackage{amssymb}
\usepackage[dvipsnames]{xcolor}
\usepackage{amscd}

\usepackage{amssymb, amsfonts, amsmath, amsthm, times}
\usepackage{epsfig}
\usepackage{lipsum}
\usepackage{graphicx}
\usepackage[sort, numbers]{natbib}
\usepackage{placeins}
\usepackage{hyperref}

\usepackage{stmaryrd}
\usepackage{textcomp}

\usepackage{authblk}
\usepackage{blindtext}

\newcommand{\too}{\longrightarrow}
\newcommand{\om}{\omega}
\newcommand{\esp}{\quad\mbox{and}\quad}

\def\br{[\;,\;]}

\newcommand{\G}{\mathfrak{g}}
\newcommand{\g}{\mathfrak{g}}

\newcommand{\ad}{{\mathrm{ad}}}
\newcommand{\Ad}{{\mathrm{Ad}}}

\newcommand{\Om}{\Omega}

\newcommand{\al}{\alpha}
\newcommand{\be}{\beta}
\newcommand{\ga}{\gamma}

\newcommand{\e}{\epsilon}

\newtheorem{remark}{Remark}

\newtheorem{Method}{Method}

\font\bb=msbm10

\def\R{\hbox{\bb R}}

\theoremstyle{definition}
\newtheorem{defi}{Definition}

\newtheorem{ex}{Examples}

\theoremstyle{plain}
\newtheorem{thm}{Theorem}[section]
\newtheorem{lem}[thm]{Lemma}
\newtheorem{prop}[thm]{Proposition}

\binoppenalty=\maxdimen
\relpenalty=\maxdimen

\title{A class of Lie racks associated to symmetric Leibniz algebras}

\author[$a$]{Hamid Abchir}
\author[$b$]{Fatima-ezzahrae Abid }
\author[$c$]{Mohamed Boucetta}
\affil[$a$]{\footnotesize Université Hassan II \\ Ecole Supérieure de Technologie \\ Route d'El Jadida Km 7, B.P. $8012, 20100$ Casablanca, Maroc \newline e-mail: h\_abchir@yahoo.com}
\affil[$b$]{\footnotesize Université Cadi-Ayyad\\ Faculté des Sciences et Techniques, BP $549$ Marrakech, Maroc\newline e-mail: abid.fatimaezzahrae@gmail.com}
\affil[$c$]{\footnotesize Université Cadi-Ayyad\\ Faculté des Sciences et Techniques, BP $549$ Marrakech, Maroc\newline e-mail: m.boucetta@uca.ac.ma}

\begin{document}

	\maketitle

\begin{abstract}
	Given a symmetric Leibniz algebra $(\mathcal{L},.)$, the product  is Lie-admissible and defines a Lie algebra bracket $\br$ on $\mathcal{L}$.  Let $G$ be  the connected and simply-connected Lie group  associated to $(\mathcal{L},\br)$.  We endow $G$ with a Lie rack structure such that the right Leibniz algebra induced on $T_eG$ is exactly $(\mathcal{L},.)$. The obtained Lie rack is said to be associated to the symmetric Leibniz algebra $(\mathcal{L},.)$.  We classify symmetric Leibniz algebras in dimension 3 and 4 and
	we determine all the associated Lie racks. Some of such Lie racks give rise to non-trivial topological quandles. We study some algebraic properties of these  quandles and we give a necessary and sufficient condition for {them}  to be quasi-trivial.  
\end{abstract}

\section{Introduction}\label{section1}
In the 1980's, Joyce \cite{joyce} and Matveev \cite{matveev} introduced the notion of  {\it quandle}.
This notion has been derived from the knot theory, in the way that the axioms of a quandle are the algebraic {interpretations} of Reidemeister moves (I,II,III) for  oriented knot diagrams \cite{elhamdadi}. The quandles provide many  knot invariants.  The fundamental quandle or knot quandle was introduced  by Joyce who showed that it is  a complete invariant of a knot (up to a weak equivalence).  Racks which are a generalization of quandles  were introduced by
Brieskorn \cite{bri} and Fenn and Rourke \cite{fenn}. Recently (see \cite{ carter1, carter2}), there has been investigations on quandles and racks from an algebraic point of view and their relationship with other algebraic structures as  Lie algebras, Leibniz algebras, Frobenius algebras, Yang Baxter equation, and Hopf algebras etc..
.

 {In 2007}, Rubinsztein introduced the notion of topological quandles \cite{rub}. Using a particular action of the braid group $B_n$ on the Cartesian product of $n$ copies of a topological quandle $(Q,\triangleright)$, he  associated the space $J_Q(L)$ of fixed points under the action of the braid $\sigma\in B_n$ for the element $\sigma$ whose closure is the oriented link $L$. The main result of the paper was that the space $J_Q(L)$ depends only on the isotopy class of the oriented link $L$. One can extend the notion of topological quandles to topological racks in a trivial way. An important subclass of  topological racks is the class  of Lie racks consisting of rack structures on smooth manifolds such that the rack operation is smooth.

The main purpose of this paper is to take advantage of a known interaction between symmetric Leibniz algebras and Lie algebras to generate families of Lie rack structures on some Lie groups of dimensions 3 and 4. We derive a family of topological quandles and we study some of their algebraic structures. Furthermore, we give a necessary and sufficient condition for such topological quandles to be quasi-trivial and then constitute link-homotopy invariants.

Let us  {give a short  overview of}  our method. We consider  $(X,\triangleright,1)$  a pointed Lie rack. That is a Lie rack with a fixed element $1\in X$, such that $x\triangleright 1=x$ and $1\triangleright x=1$, for each $x\in X$. It is known (see \cite{kinyon}) that the tangent space $T_1X$ has a structure of right Leibniz algebra. The problem of integrating Leibniz algebras to pointed Lie rack{\color{red}s} was formulated by J. -L. Loday in \cite{loday}. It consists in finding a generalization of the Lie's third theorem for Leibniz algebras. There are only partial answers to this problem (see \cite{covez, bord}). However, S. Benayadi and M. Bordemann \cite{benbor} gave a natural method for integrating  symmetric Leibniz algebras, which are both right and left Leibniz algebras. This method is based on the characterization of symmetric Leibniz algebras given in \cite{saidbar}. More precisely, given a symmetric Leibniz algebra $(\mathcal{L},.)$, the product  is Lie-admissible and defines a Lie algebra bracket $\br$ on $\mathcal{L}$.  Let $G$ be  the connected and simply-connected Lie group  associated to $(\mathcal{L},\br)$.  Then, naturally one can build on $G$  a Lie rack structure such that the right Leibniz algebra on $T_eG$ is exactly $(\mathcal{L},.)$. The obtained Lie rack is said to be associated to the symmetric Leibniz algebra $(\mathcal{L},.)$. Having this method in mind, we determine all symmetric Leibniz algebras in dimension 3 and 4, up to an isomorphism, and for each of them we build the associated Lie rack. We get a family of Lie racks  and some topological quandles.  We study some algebraic properties of these  quandles. 

Our exposition is organized as follows. The Section~\ref{section2} is devoted to preliminaries. We recall the notions of Lie racks, quandles and Leibniz algebras. In Section~\ref{section3}, we state our main result which introduces a Lie rack structure on the connected simply connected Lie group associated to the underlying Lie algebra of a given symmetric Leibniz algebra. We also investigate some algebraic properties of the associated topological quandle. In section~\ref{section4}, we {first} give  all symmetric Leibniz {algebras} of dimension 3 and 4, and then we apply our method~\ref{me} to generate all  the associated Lie racks. In Section~\ref{section5}, we study some  algebraic properties of the derived topological quandles. In Section~\ref{section6},  we give an example of explicit calculations in dimension 4. 

\section{Preliminaries}\label{section2}

\subsection{Lie racks and topological quandles}
\begin{defi}\label{qd}
	\begin{enumerate}
		\item A rack is a  non-empty set $\mathit{X}$ together with a map $\rhd \, : \mathit{X} \times \mathit{X} \longrightarrow \mathit{X}$, $(x,\, y) \mapsto x \rhd y$ such that
		\begin{itemize}
			\item for any fixed element $x \in \mathit{X}$, the map $\mathrm{R}_{x} \, : \mathit{X} \longrightarrow \mathit{X}$, $y \mapsto y \rhd x $ is a bijection,
			\item for any $x,\, y,\, z \in \mathit{X}$, we have  $ (x \rhd y) \rhd z = (x \rhd z)\rhd (y \rhd z)$ {(right self-distributivity)}.
		\end{itemize}  
		\item A rack $\mathit{X}$ is called pointed, if there exists a distinguished element $1 \, \in \mathit{X}$ such that
		$$ x \rhd 1 =\mathrm{R}_{1}(x)=id_{\mathit{X}}(x)=x  \mbox{ and }  1 \rhd x =\mathrm{R}_{x}(1)=1, \, {{\rm for\,\, each}} \,\, x\in \mathit{X}.$$
		\item A rack $\mathit{X}$ is called a quandle if, for any $x\in \mathit{X},\; x \rhd x =x$.
		\item A quandle $\mathit{X}$ is called a Kei if, for any $x,\, y \in \mathit{X},\; (y \rhd x) \rhd x = y$,  i.e.,   $\mathrm{R}_{x}$ is an involution.
		\item A topological quandle (rack) is a  topological space $\mathit{X}$ with a  quandle {(rack)} structure such that  the product
		$\rhd \, : \mathit{X} \times \mathit{X} \longrightarrow \mathit{X}$ is continuous and, for all {$x \in \mathit{X}$, $\mathrm{R}_{x} \, : \mathit{X} \longrightarrow \mathit{X}$, $y \longmapsto y \rhd x$} is a homeomorphism.
		\item A Lie rack is a smooth manifold $\mathit{X}$ with a  rack structure such that  the product
		$\rhd \, : \mathit{X} \times \mathit{X} \longrightarrow \mathit{X}$ is smooth and, for all $x \in \mathit{X}$, $\mathrm{R}_{x} \, : \mathit{X} \longrightarrow \mathit{X}$, $y \longmapsto y \rhd x$ is a diffeomorphism.
	\end{enumerate}
\end{defi}
When $\mathit{X}$ is a rack, sometimes we write $ y \rhd^{-1} x:=\mathrm{R}^{-1}_{x}(y)$.

\begin{remark} {The rack defined above is said to be a right distributive rack}.  There is also the notion of left distributive rack which is equivalent. {In the following, we will consider the right version unless otherwise stated.}
	
\end{remark} 

\begin{ex}
	\begin{itemize}
		\item Any non-empty set $\mathit{X}$ equipped with the operation $x \rhd y\,:= x $ for any $x,\, y \in  \mathit{X}$ is a kei, which is called the trivial kei, the trivial quandle or the trivial rack.
		\item Let $\mathit{G}$ be a group. Then $\mathit{G}$ is a quandle under the operation of conjugation, i.e.
			$$ h \rhd g\, =\, g^{-1}hg \quad {\rm for\ all}\, g,h \, \in\ \mathit{G}.$$
		We denote this quandle by $\mathrm{Conj}(G)$.
		\item Let $\mathbb{Z}_{n}$ be the ring of integers modulo $n \in \mathbb{N}^{\ast}$. For any $x,\,y \in  \mathbb{Z}_{n}$, we define the operation $x \rhd y\,= 2y-x$. The pair $(\mathbb{Z}_{n},\, \rhd )$ is a quandle which is called the \textit{dihedral {quandle}} and is denoted by $\mathit{R}_{n}$.
		\item Let $\mathbb{Z}[t,\, t^{-1}]$ be the ring of Laurent polynomials in the variable $t$. Let $\mathit{M}$ be a $\mathbb{Z}[t,\, t^{-1} ]$-module. The operation $x \rhd y = tx + (1-t)y$ for any {$x,y \in \mathit{M}$}, makes $\mathit{M}$ into a quandle called the Alexander quandle.
	\end{itemize}
\end{ex}

{A map $f$ between two racks $(X_1,\rhd_1)$ and $(X_2,\rhd_2)$ is a \textit{rack homomorphism} if $f$ preserves the rack operations, i.e., $f(x\rhd_1 y)=f(x)\rhd_2 f(y)$ for all $x,y$ $\in X_1$. If furthermore $f$ is a bijection it is called an isomorphism of racks. In particular, a bijective rack homomorphism $f: \mathit{X} \too \mathit{X}$ is called a \textit{rack automorphism}. One can define a quandle homomorphism in exactly the same way.}
\begin{remark}
	It is easy to see that for any rack $X$ and any $x\in X$, the right translation $\mathrm{R}_x$ is a rack automorphism.
\end{remark}
{The last remark allows to show the following proposition.}
\begin{prop}\label{idemp}
	Let {$(\mathit{X}, \rhd)$} be a rack and  $Q(\mathit{X})$ be the set of its idempotents,
	\[Q(X) = \{ x \in X \;,\; x \rhd x = x \},\]
	then $(Q(\mathit{X}), \rhd)$ is a quandle. In particular, if $(\mathit{X}, \rhd)$  {is} a Lie rack {then} $(Q(\mathit{X}), \rhd)$ is a topological quandle.
\end{prop}
\begin{proof}
	We note first that $Q(\mathit{X})$ is closed by the binary operation $\rhd$. Indeed,
	if $x, y \in Q(\mathit{X})$, then  \[ (x \rhd y) \rhd (x\rhd y) = (x \rhd x) \rhd y = x \rhd y.\] The right distributivity is obviously satisfied.
	
	For any $x, y \in Q(\mathit{X})$, the restriction of the right translation $\mathrm{R}_{y}$ to $Q(\mathit{X})$ is injective by assumption. Actually
	$\mathrm{R}_{y}$ is bijective when considered as a map defined on $\mathit{X}$. Let $y, z \in Q(\mathit{X})$. There exists a unique $x \in  \mathit{X}$ such that 
	$\mathrm{R}_{y}(x) = z$ and then, $x = \mathrm{R}_{y}^{-1}(z) = z \rhd^{-1} y$. Since $\mathrm{R}_{y}^{-1}$ is a quandle morphism, we have
	\[ x \rhd x = \mathrm{R}_{y}^{-1}(z)\rhd \mathrm{R}_{y}^{-1}(z) = \mathrm{R}_{y}^{-1}(z \rhd  z) = \mathrm{R}_{y}^{-1}(z) = x\]
	Then the unique $x \in  \mathit{X}$ such that $x = \mathrm{R}_{y}^{-1}(z)$ belongs to $Q(\mathit{X})$. This ends the proof. 
\end{proof}

Furthermore, it is known that the set of all rack automorphisms of $\mathit{X}$ forms a group denoted {by} $\mathrm{Aut}(\mathit{X})$. The group of inner automorphisms $\mathrm{Inn}(\mathit{X})$ generated by all bijections $\mathrm{R}_{x}$ is a normal subgroup of $\mathrm{Aut}(\mathit{X})$. Then, the map 
$$\begin{array}{c c c c }
\mathrm{R} \; : &  \mathit{X} & \too  & \mathrm{Inn}(\mathit{X})\\
& x  & \longmapsto & \mathrm{R}_{x}
\end{array}$$
induces a right action of the group $\mathrm{Inn}(\mathit{X})$ on $\mathit{X}$. The orbit $\Om(x)$ of an element $x\in \mathit{X}$ is given by
{\[\Om(x){\color{red}\xout{ :  }}=\{\varphi(x), \varphi \in \mathrm{Inn}(\mathit{X})  \}=\{\mathrm{R}_y(x),\, y\in X\}. \]}
Note that the notion of inner automorphisms of quandles is similarly defined. 

{Let $(Q,\rhd)$ be a quandle}. The subset $Z(Q){\color{red}\sout{ : }}=\{x \in Q\;:\; x \rhd y=x \quad \forall  y \in Q \}$  is called the \textit{center} of the quandle $Q$. In particular, if $Q = \mathrm{Conj}(\mathit{G})$, then the center $Z(Q)$ of the quandle $Q$ coincides with the center $Z(\mathit{G})$ of the group $\mathit{G}$. 

We end this section by recalling the definition of three classes of quandles (see for instance~\cite{inoue, davidl, medial}). We will show in the last section that some quandles we will obtain are in these classes. 
\begin{defi}\label{def}\begin{enumerate}
		\item A quandle $(Q,\rhd )$ is called \textit{quasi-trivial} if $x \rhd \varphi(x) = x$  for any
		$x \in Q $ and $\varphi \in \mathrm{Inn}(Q)$.
		This  is equivalent to {$x \rhd y =x $ for all $x,y\in\Om(x)$}.
		\item A quandle $Q$ is called \textit{medial}, if for any $x,y,z,w \in Q$, we have 
		\[(x\rhd y)\rhd (z\rhd w)= (x\rhd z)\rhd (y\rhd w). \]
	    \end{enumerate}
\end{defi}

\subsection{Symmetric Leibniz algebras}
In this subsection, we recall the definition of a Leibniz algebra with an emphasis on the structure of a symmetric Leibniz algebra for which we give a useful characterization and some immediate properties. For more details on Leibniz algebras one can see \cite{covez,loday}.\\
Let $(\mathfrak{L}, .)$ be an  algebra. For any $u\in \mathfrak{L}$, we denote by $L_{u}$, respectively $R_{u}$, the two endomorphisms of the vector space $\mathfrak{L}$ defined by $L_{u}(v)= u.v$ and $R_{u}(v)= v.u$, $\forall v \in \mathfrak{L}$. The maps $L_{u}$ and $R_{u}$ are respectively called the left translation and the right translation by $u$.
\begin{defi} 
	\begin{enumerate}
		\item An algebra $(\mathfrak{L}, .)$ is said to be a left Leibniz algebra, if for each $u \in \mathfrak{L}$, the left translation ${L}_{u}$ is a derivation. That is, for any $v,\, w \in \mathfrak{L}$ we have the following identity
		\begin{equation}\label{1}
		u.(v.w)=(u.v).w + v.(u.w).
		\end{equation}
		\item An algebra $(\mathfrak{L}, .)$ is said to be a right Leibniz algebra, if for each $u \in \mathfrak{L}$, the right translation ${R}_{u}$ is a derivation. That is, for any $v,\, w \in \mathfrak{L}$ we have the following identity
		\begin{equation}\label{3}
		(v.w).u= (v.u).w + v.(w.u).
		\end{equation}
		\item If $(\mathfrak{L},.)$ is both a right and a left Leibniz algebra then it is called a symmetric Leibniz algebra.
	\end{enumerate}
\end{defi}

Any Lie algebra is a symmetric Leibniz algebra. However, the class of symmetric Leibniz algebras is far more bigger than the class of Lie algebras as we will see later.\\
Let $\mathfrak{L}$ be a real vector space equipped with a bilinear map $ \cdot : 
\mathfrak{L}\times \mathfrak{L} \longrightarrow \mathfrak{L}$. Let $[\:,\:]$ and $\circ$ be respectively its antisymmetric and symmetric parts. For all $u,\,v\: \in \mathfrak{L}$, they are defined by:
\begin{equation}
[u,v]= \frac{1}{2}(u.v-v.u) \quad and \quad u \circ v = \frac{1}{2}(u.v+v.u).
\end{equation}
Thus 
\begin{equation}\label{pdl}
u.v = [u,v] + u \circ v.
\end{equation}

The following proposition gives a useful characterization of symmetric Leibniz algebras (see~[\cite{saidbar} , Proposition 2.11]).
\begin{prop}\label{3.1}
	Let $(\mathfrak{L},.)$ be an algebra. The following assertions are equivalent:
	\begin{itemize}
		\item[1] $(\mathfrak{L},.)$ is a symmetric Leibniz algebra.
		\item[2]The following conditions hold :
		\begin{itemize}
			\item[(a)] $(\mathfrak{L},[\,,\,])$ is a Lie algebra.
			\item[(b)] For any $u,v \in \mathfrak{L}$, $u\circ v$ belongs to the center of $(\mathfrak{L},[\,,\,])$.
			\item[(c)] For any {$u,v,w \in \mathfrak{L}$, $([u,v])\circ w=0 \text{  and  }  (u\circ v)\circ w=0$}. 
		\end{itemize}
	\end{itemize}
\end{prop}

According to this proposition, any symmetric Leibniz algebra is given by a Lie algebra $(\mathfrak{L},\, [\, ,\, ])$ and a bilinear symmetric form  $\omega \, : \mathfrak{L} \times \mathfrak{L} \longrightarrow Z(\mathfrak{L})$ where $Z(\mathfrak{L})$ is the center of the Lie algebra, such that, for any $u,\, v,\, w \in \mathfrak{L}$,
\begin{equation}\label{eq}
\omega([u, v], w) = \omega(\omega(u, v), w) = 0.
\end{equation}
Then the product of the  symmetric Leibniz algebra is given by
\begin{equation}\label{symprod}
u.v = [u, v] + \omega(u, v), \quad u, v \in \mathfrak{L}.
\end{equation}

Note that if $Z(\mathfrak{L}) = 0 $ or $[ \mathfrak{L}, \, \mathfrak{L}] = \mathfrak{L}$ then the solutions of ~\ref{eq} are trivial.

The following proposition is easy to prove.

\begin{prop}\label{equi} Let $(\G,\br)$ a Lie algebra and $\om$ and $\mu$ two solutions of~\ref{eq}. Then $(\G,\bullet_\om)$ is isomorphic to $(\G,\bullet_\mu)$ (as symmetric Lie algebras) if and only if there exists an automorphism $A$ of $(\G,\br)$ such that
	\[ \mu(u,v)=A^{-1}\om(Au,Av). \]\end{prop}

\section{Lie racks and {topological quandles} associated to symmetric Leibniz algebras}\label{section3}

Let $(\mathit{X},1)$ be a pointed Lie rack with left distributivity. Kinyon showed in \cite{kinyon} that the tangent space $T_1\mathit{X}$ carries a structure of left Leibniz algebra. In what follows, we show that, in the same way, one can get a structure of a right Leibniz algebra on the tangent space $T_1\mathit{X}$ if the pointed Lie rack $(\mathit{X},1)$ is considered with  right distributivity.\\
For each $x\, \in \mathit{X}$, $\mathrm{R}_{x}(1)= 1 $. We consider the linear map $${\Ad_x}=T_{1}\mathrm{R}_{x}\,:T_1\mathit{X}\longrightarrow T_1\mathit{X}.$$
We have 
\[\Ad_{x\rhd y} = \Ad_{x}\circ \Ad_{y} \circ \Ad^{-1}_{x}.\]
Thus $\Ad\, : \mathit{X}\longrightarrow \mathrm{GL}(T_1\mathit{X})$ is an homomorphism of Lie racks.
If we put $$ u.v:= \ad_{u}(v)=\frac{d}{dt}_{| t=0} \Ad_{(c(t))}(u), \; \,\forall \, u,v \in T_{1}\mathit{X}{\color{red},}$$
where $c \, : ]-\epsilon,\, \epsilon [ \longrightarrow \mathit{X}$ is a smooth path in $\mathit{X}$ such that $c(0)=1 \mbox{ and } c'(0)= v$. We have the following theorem.

\begin{thm}[\cite{kinyon}]\label{tang} Let $(\mathit{X},1)$ be a pointed Lie rack. Then the tangent space {$T_{1}\mathit{X}$} endowed with the product $$ u.v:= \ad_{u}(v)=\frac{d}{dt}_{| t=0} \Ad_{(c(t))}(u), \; \,\forall \, u,\,v \in T_{1}\mathit{X},$$
	is a right Leibniz algebra. Moreover, if $\mathit{X}=\mathrm{Conj}(\mathit{G})$ where $\mathit{G}$ is a Lie group  then $(T_1\mathit{X},\,.)$ is the Lie algebra of $\mathit{G}$. 
\end{thm} 
The problem of integrating {a} Leibniz algebra into {a} pointed Lie {r}ack was first formulated by J.-L. Loday in\cite{loday}. Though
until now there is no natural answer to this problem, but there are many partial results. In \cite{kinyon}, Kinyon gave a
positive answer for split Leibniz algebras. In\cite{covez}, S. Covez gave a local answer to the integration
problem. He showed that every Leibniz algebra can be integrated into {a} local augmented Lie
rack. In \cite{bord}, Bordemann gave a global process of integrating Leibniz algebras but this process
is, unfortunately, not functorial. For our purpose, there is a functorial process of integrating
symmetric Leibniz algebras which was communicated to us privately by S. Benayadi and M. Bordemann and we will give it now in details.

Let $(\mathfrak{L},\,.)$ be a symmetric Leibniz algebra. According to Proposition~\eqref{3.1}, there exists a
Lie bracket $\br$ on $\mathfrak{L}$ and a bilinear symmetric form $\omega : \mathfrak{L}\times\mathfrak{L} \longrightarrow Z(\mathfrak{L})$, where $Z(\mathfrak{L})$ is the center of $(\mathfrak{L},\;\br)$, such that,
\[u.v = [u, v] + \omega(u, v) \; \mbox{for all }\; u, v \in  \mathfrak{L}\]
and $\omega$ satisfies \eqref{eq}.

\begin{Method}\label{me} Denote by $\mathfrak{L}.\mathfrak{L}= \text{span}\{u.v,\, u,\,v \in \mathfrak{L} \}$, $\mathfrak{a}= \mathfrak{L}/\mathfrak{L}.\mathfrak{L}$, $q: \mathfrak{L} \longrightarrow \mathfrak{a}$ and define $\beta: \mathfrak{a} \times \mathfrak{a} \longrightarrow \mathfrak{L}$ given by 
	\begin{equation*}
	\beta (q(u),\, q(v)) =\omega (u, v).
	\end{equation*}
	By virtue of  \eqref{eq}, $\beta$ is well defined. Moreover, since $[u, v] =\frac{1}{2}(u.v - v.u)$, $q$ is a Lie algebra homomorphism when $\mathfrak{a}$ is considered as an abelian Lie algebra.
	Consider $\mathit{G}$ the connected and simply connected Lie group whose Lie algebra is $(\mathfrak{L},\,[\,,\,])$ and $\mbox{exp}\,: \mathfrak{L} \longrightarrow \mathit{G}$ its exponential. Hence there exists an homomorphism of Lie groups $\kappa \,: \mathit{G} \longrightarrow \mathfrak{a}$ such that $d_{e} \kappa \,=\, q$. Finally, consider $ \chi \, :\mathit{G} \times \mathit{G} \longrightarrow \mathit{G}$ given by 
	\begin{equation*}
	\chi (h,\,g)\,=\, \mbox{exp} (\beta(\kappa(h),\, \kappa(g))){\ {\rm for\ all}\  h,g\in G}.
	\end{equation*}
	Define now on $\mathit{G}$ the binary product by putting
	\begin{equation}\label{sb}
	{h \rhd g \, :=\,g^{-1}hg \chi (h,\,g)\ {\ {\rm for\ all}\  h,g\in G}.}
	\end{equation}\end{Method}
{ We show the result obtained by S. Benayadi and M. Boredmann.}
\begin{thm}
	{ Let $(\mathfrak{L},\,.)$ be a symmetric Leibniz algebra}. Let $\mathit{G}$ be the connected simply connected Lie group associated to the underlying Lie algebra endowed with the binary operation $\rhd$ defined by \eqref{sb}.
		Then $(\mathit{G},\rhd)$ is a {pointed} Lie rack  {whose}  associated right Leibniz algebra is  {exactly} $(\mathfrak{L},\,.)$.
\end{thm}

\begin{proof}
	 {At first, we note} that the map $\chi$  {satisfies} the following properties for all elements $h,h_1,h_2,h_3\in \mathit{G}$:   
	{
		\begin{itemize}
			\item The map $\chi$ is symmetric, since $\beta$ is {symmetric}.
			\item  $\chi(h,\,1)=1=\chi(1,\,h)$ because $\kappa(1)=0$ and $\beta$ is bilinear.
			\item 
			\begin{eqnarray*}
				\chi(h_1h_2,\,h_3)&=& \mbox{exp}(\beta(\kappa(h_1h_2),\,\kappa(h_3)))= \mbox{exp}(\beta(\kappa(h_1)+\kappa(h_2),\,\kappa(h_3)))\\
				&=& \mbox{exp}(\beta(\kappa(h_1),\,\kappa(h_3)))\mbox{exp}(\beta(\kappa(h_2),\,\kappa(h_3)))\\
				&=&\chi(h_1,\,h_3)\chi(h_2,\,h_3)
			\end{eqnarray*}
			\item  {Since the map $\beta$ takes} its values  in the centre of $\mathfrak{L}$, it  follows that, $$1=\chi(h^{-1}_{1}h_1,\,h_2)=\chi(h^{-1}_{1},\,h_2)\chi(h_1,\,h_2)$$ and so $\chi(h^{-1}_{1},\,h_2)=\chi(h_{1},\,h_2)^{-1}$.
			\item  Since $\kappa$ is a morphism  {between} connected simply {connected} Lie groups, then for all $x\in \mathfrak{L}$ we have $\kappa(\mbox{exp}(x))= \mbox{exp}_{a}(T_{1}\kappa(x))$. The latter {is} equal to $q(x)$ because the exponential map of the vector space Lie group $\mathfrak{a}$ is the identity.{ Then} 
			$$\kappa(\chi(h_1,\,h_2))=\kappa(\mbox{exp}(\beta(\kappa(h_1),\,\kappa(h_2))))=q(\beta(\kappa(h_1),\,\kappa(h_2)))=0,$$
			{hence} 
			$$ \chi(h_1\chi(h_2,\,h_3),\,h_4)= \chi(h_1,\,h_4).$$
		\end{itemize}
	}
	{Now we will use those properties to show the theorem}.
	
	{First, it is easy to see}  that the binary operation $\rhd : \mathit{G}\times \mathit{G} \longrightarrow \mathit{G}$ is a smooth map. {Then} for any $g \in \mathit{G}$, the map $\mathrm{R}_{g}\: : \mathit{G}\longrightarrow \mathit{G}$ {which} sends $h$ to $h \rhd g$ is invertible with smooth inverse {$\mathrm{R}^{-1}_{g}= \mathrm{R}_{g^{-1}}$.} {This} follows easily from the identity $\chi( h\rhd g,\, h_{1})= \chi(h,\, h_{1})$ for any $ h,\,g,\,h_{1} \in \mathit{G}$. 
	
	 {Let us} show that {the} self-distributivity {condition}  is satisfied. {For that we will use the following identity. If} $z \in \mathit{G}$ {writes} $z= \chi(h_1,\,h_2)$ {for $h_1,h_2\in\mathit{G}$, then} 
	\begin{eqnarray*}
		h_1 \rhd (zh_{2})&=& (zh_{2})^{-1}h_1 (zh_{2})\chi(h_1,\,zh_{2})\\
		&=& h^{-1}_2 z^{-1}h_1 zh_{2}\chi(h_1,\,h_{2})\\
		&=& h^{-1}_2 h_1h_2 \chi(h_1,\,h_2)\\
		&=& h_1 \rhd h_2.
	\end{eqnarray*}
	On {the} one hand we have
	\begin{eqnarray*}
		(h_1 \rhd h_2)\rhd h_3 &=& (h^{-1}_2h_1h_2\chi(h_1,\,h_2))\rhd  h_3\\
		&=& h^{-1}_3h^{-1}_2 h_1 h_2h_3\chi(h_1,\,h_2)\chi(h^{-1}_2h_1h_2\chi(h_1,\,h_2),\,h_3)\\
		&=& (h_2h_3)^{-1}h_1(h_2h_3)\chi(h_1,\,h_2)\chi(h_1,\,h_3)\\
		&=& (h_2h_3)^{-1}h_1(h_2h_3)\chi(h_1,\,h_2h_3)\\
		&=& h_1 \rhd (h_2h_3).
	\end{eqnarray*}
	 On the other hand 
	\begin{eqnarray*}
		(h_1 \rhd h_3)\rhd (h_2 \rhd h_3)&=& h_1 \rhd (h_3(h_2 \rhd h_3))\\
		&=& h_1 \rhd (h_2h_3\chi(h_2,\,h_3))\\
		&=&  \chi(h_2,\,h_3)^{-1}h^{-1}_3h^{-1}_2 h_1 h_2h_3\chi(h_2,\,h_3)\chi(h_1,\,h_2h_3\chi(h_2,\,h_3))\\
		&=& (h_2h_3)^{-1}h_1h_2h_3\chi(h_1,\,h_3h_2)\\
		&=& h_1 \rhd (h_2h_3),
	\end{eqnarray*}
	 {This proves} the self-distributivity  condition. Moreover we  {have} 
	$$h_1 \rhd 1  =  h_1\chi (1,\,h_1) =h_1\esp 1 \rhd  h_1  =  h^{-1}_1h_1\chi (1,\,h_1) =1$$
	 {Finally $(\mathit{G},\rhd)$ is a pointed} Lie rack.\\
	{For the last claim in the theorem}  {we must show} that the corresponding Leibniz product on $\mathfrak{L}$ is {exactly} the  product  we started with {in $\mathfrak{L}$}. Indeed, according to Theorem~\eqref{tang}, we get for all $u,v \in \mathfrak{L}$ and $x \in \mathit{G}$ :  
	\begin{eqnarray*}
		\frac{\partial}{\partial t}_{|t=0}\mathrm{R}_{x}(\exp(tu))&=& \frac{\partial}{\partial t }_{|t=0} (\exp(tu) \rhd x),\\
		&=& \frac{\partial}{\partial t}_{|t=0}(x^{-1}\exp(tu)x \chi(x,\, \exp(tu)),\\
		&=& \frac{\partial}{\partial t}_{|t=0}(\exp(tx^{-1}ux) \exp(\beta(\kappa(x),\, \kappa(\exp(tu))))),\\
		&=& \frac{\partial}{\partial t}_{|t=0}(\exp(tAd_{x^{-1}}(u)) \exp(\beta(\kappa(x),\,tq(u)))),\\
		&=& Ad_{x^{-1}}(u)\,+\, \beta(\kappa(x),\,q(u)).
	\end{eqnarray*}
	Replacing $x$ by the curve $t\: \longmapsto \exp(tv)$, we  obtain 
	\begin{eqnarray*}
		\frac{\partial}{\partial t}_{|t=0}T_{1}\mathrm{R}_{\exp(tv)}(u))&=& \frac{\partial}{\partial t }_{|t=0} (Ad_{\exp(-tv)}(u)\,+\, \beta(\kappa(\exp(tv)),\,q(u))),\\
		&=& \frac{\partial}{\partial t}_{|t=0}(Ad_{\exp(-tv)}(u))\,+\, \beta(q(v),\,q(u)) ,\\
		&=& [-v,\,u]\,+\, v \circ u,\\
		&=& [u,\, v]\, + u \circ v= u.v.
	\end{eqnarray*}
	 {This establishes} the formula~(\ref{pdl}) and completes the proof.
\end{proof}
{ We say that the obtained Lie rack $(\mathit{G},\, \rhd )$ is associated to the symmetric Leibniz algebra $(\mathfrak{L},\,.)$.}

The Lie racks associated to symmetric Leibniz algebras give rise to a class of topological quandles in the following way. Let $(\mathfrak{L},\,.)$ be a symmetric Leibniz algebra and $(\mathit{G},\, \rhd )$ be the associated Lie rack. According to Proposition \ref{idemp},
$$\mathit{Q}((\mathit{G},\rhd))=\{g\in G\ ;\ \chi(g,g)=1_G\}$$
is a topological quandle.

\section{Lie racks associated to symmetric Leibniz algebras in dimensions 3 and 4}\label{section4}

In this section, by using Proposition \ref{3.1} we determine first all the symmetric Leibniz algebras of dimension 3 and 4, up to an isomorphism and, for each of them, we use  Method \ref{me} described in the last section to build the associated Lie racks.

\subsection{Symmetric Leibniz algebras of dimension 3 and 4}
 
We proceed in the following way:
\begin{enumerate}\label{om}
\item We pick a Lie algebra $\G$ with non trivial center in the list of \cite{biggs}.
	\item By a direct computation, we determine the symmetric forms $\om$ satisfying \eqref{eq}.
	\item In the spirit of Proposition \ref{equi}, we act by the group of automorphisms of $\G$ on the obtained $\om$  to reduce the parameters.
	
\end{enumerate}By doing so, we get for any Lie algebra $\G$ of dimension 3 or 4 with non trivial center all  non equivalent symmetric Leibniz structures { for which} $\G$ { is} the underlying Lie algebra. {In} the last section, we give an example of detailed computations. The results are summarized in Table {\eqref{tab1}}.\\

\subsection{Lie racks}\label{quand}
In this subsection, we determine by using Method \ref{me} the  Lie racks  associated to the symmetric Leibniz { algebras} determined in the last subsection. Then we give  the associated topological quandles defined in Proposition \ref{idemp}. In the  { last section}, we   {explicit} the computations   {for a particular} example.

\begin{itemize}
\item[$\bullet$ $\g_{3,1}$.] The associated simply-connected Lie group is given by
	$$\mathit{G}_{3,1} = \displaystyle\left\{  
	\begin{bmatrix}
	1 & y & x \\ 
	0 & 1 & z \\ 
	0 & 0 & 1
	\end{bmatrix},
	x,\,y,\,z,\,\in \mathbb{R}\displaystyle\right\}.$$
\begin{enumerate}\item The Lie rack structure associated to $\g_{3,1}^1$ and the associated topological quandle are respectively
\begin{footnotesize}\[M(x,\,y,\,z) \rhd_1 M(a,\,b,\,c)\, =\begin{bmatrix} 1&y&yb+2\,cy+zc+x\\ \noalign{\medskip}0&1&
	z\\ \noalign{\medskip}0&0&1\end{bmatrix} \]\end{footnotesize}
	\[Q(\mathit{G}_{3,1})^1=
	\{M(a,b,c)\in \mathit{G}_{3,1}|  \; b=-c \}.\]
	\item The Lie rack structure associated to $\g_{3,1}^2$ and the associated topological quandle are respectively
 \begin{footnotesize}	\begin{equation*}
   M(x,\,y,\,z) \rhd_2 M(a,\,b,\,c)\, = \begin{bmatrix} 1&y&-bz+cy+zc+x\\ \noalign{\medskip}0&1&z
  \\ \noalign{\medskip}0&0&1\end{bmatrix} 
  \end{equation*}\end{footnotesize}
  \[Q(\mathit{G}_{3,1})^2 =
  \{M(a,b,c)\in \mathit{G}_{3,1}|  \; c=0 \}.\]
\item The Lie rack structure associated to $\g_{3,1}^3$ and the associated topological quandle are respectively
\begin{footnotesize}\begin{equation*}
M(x,\,y,\,z) \rhd_3 M(a,\,b,\,c)\, =\begin{bmatrix} 1&y&\epsilon\,zc+yb-bz+cy+x
\\ \noalign{\medskip}0&1&z\\ \noalign{\medskip}0&0&1\end{bmatrix}
\end{equation*}\end{footnotesize}
 \[Q(\mathit{G}_{3,1})^3=
\{M(a,b,c)\in \mathit{G}_{3,1}|  \; bc=0  \}.\]
\item The Lie rack structure associated to $\g_{3,1}^4$ and the associated topological quandle are respectively
\begin{footnotesize}\begin{equation*}
M(x,\,y,\,z) \rhd_4 M(a,\,b,\,c)\, =\begin{bmatrix} 1&y&\ga bz+\ga cy-bz+cy+x
\\ \noalign{\medskip}0&1&z\\ \noalign{\medskip}0&0&1\end{bmatrix}
\end{equation*}\end{footnotesize}
\[Q(\mathit{G}_{3,1})^4 =
\{M(a,b,c)\in \mathit{G}_{3,1}|  \; bc=0  \}.\]
	
\end{enumerate}	

\item[$\bullet$ $\g_{2,1}\oplus 2\g_{1}$. ] The associated simply-connected Lie group  is given by 
$$\mathit{G}_{2,1}\times \R^2 = \displaystyle\left\{  
\begin{bmatrix}
1 & 0 &0 & 0\\ 
w  & e^{-x} &0 &0 \\
0 & 0&e^y &0\\
0& 0&0 & e^z 
\end{bmatrix},
w,x,\,y, z \,\in \mathbb{R}\displaystyle\right\}.$$
\begin{enumerate}\item The Lie rack structure associated to $(\g_{2,1}\oplus 2\g_{1})^1$ and the associated topological quandle are defined by the \textbf{conjugation} operation on $\mathit{G}_{2,1}\times \R^2$.
\item  The Lie rack structure associated to  $(\g_{2,1}\oplus 2\g_{1})^2$ and the associated topological quandle are respectively
\begin{footnotesize}
	\[ M(w,\,x,\,y,\,z)\rhd_2 M(t,\,a,\,b,\,c)=\begin{bmatrix}1&0&0&0\\ \noalign{\medskip}-t{{\rm e}^{a
}}+{{\rm e}^{a}}w+t{{\rm e}^{a-x}}&{{\rm e}^{-x}}&0&0
\\ \noalign{\medskip}0&0&{{\rm e}^{ax+y}}&0\\ \noalign{\medskip}0&0&0&
{{\rm e}^{ax+z}}\end{bmatrix}\] \end{footnotesize}
\[Q(\mathit{G}_{2,1}\times \R^2)^2 =
\{M(t,a,b,c)\in \mathit{G}_{2,1}\times \R^2|  \; a=0 \}.\]
\item Lie rack structure associated to  $(\g_{2,1}\oplus 2\g_{1})^3$ and the associated topological quandle
\begin{footnotesize}\[ M(w,\,x,\,y,\,z)\rhd_3 M(t,\,a,\,b,\,c)= \begin{bmatrix} 1&0&0&0\\ \noalign{\medskip}-t{{\rm e}^{a
 }}+{{\rm e}^{a}}w+t{{\rm e}^{a-x}}&{{\rm e}^{-x}}&0&0
 \\ \noalign{\medskip}0&0&{{\rm e}^{ax+y}}&0\\ \noalign{\medskip}0&0&0&
 {{\rm e}^{z}}\end{bmatrix} \] \end{footnotesize}
 \[Q(\mathit{G}_{2,1}\times \R^2)^3=
 \{M(t,a,b,c)\in \mathit{G}_{2,1}\times \R^2|  \; a=0 \}.\]
 \item  The Lie rack structure associated to  $(\g_{2,1}\oplus 2\g_{1})^4$ and the associated topological quandle are respectively
\begin{footnotesize}\[ M(w,\,x,\,y,\,z)\rhd_4 M(t,\,a,\,b,\,c)= \begin{bmatrix} 1&0&0&0\\ \noalign{\medskip}-t{{\rm e}^{a
 }}+{{\rm e}^{a}}w+t{{\rm e}^{a-x}}&{{\rm e}^{-x}}&0&0
 \\ \noalign{\medskip}0&0&{{\rm e}^{xc+za+y}}&0\\ \noalign{\medskip}0&0&0&
 {{\rm e}^{z}}\end{bmatrix} \] \end{footnotesize}
  \[Q(\mathit{G}_{2,1}\times \R^2)^4 =
 \{M(t,a,b,c)\in \mathit{G}_{2,1}\times \R^2|  \; ac=0  \}.\]
 \item  The Lie rack structure associated to  $(\g_{2,1}\oplus 2\g_{1})^5$ and the associated topological quandle are respectively
\begin{footnotesize} \[ M(w,\,x,\,y,\,z)\rhd_5 M(t,\,a,\,b,\,c)= \begin{bmatrix} 1&0&0&0\\ \noalign{\medskip}-t{{\rm e}^{a
 }}+{{\rm e}^{a}}w+t{{\rm e}^{a-x}}&{{\rm e}^{-x}}&0&0
 \\ \noalign{\medskip}0&0&{{\rm e}^{\ga xa+y}}&0\\ \noalign{\medskip}0&0&0&
 {{\rm e}^{z}}\end{bmatrix} \] \end{footnotesize}
 \[Q(\mathit{G}_{2,1}\times \R^2)^5 =
 \{M(t,a,b,c)\in\mathit{G}_{2,1}\times \R^2| \; \ga a=0  \}.\]
\end{enumerate}

\item[$\bullet$ $\g_{3,1}\oplus \g_{1}$.] The associated simply-connected Lie group is given by 
$$\mathit{G}_{3,1}\times \R = \displaystyle\left\{  
\begin{bmatrix}
1 & x& w &0\\ 
0 & 1 &y& 0 \\ 
0 & 0 & 1& 0\\
0& 0& 0&e^z 
\end{bmatrix},w,x,\,y,\,z,\,\in \mathbb{R}\displaystyle\right\}$$
\begin{enumerate}
	\item The Lie rack structure associated to $(\g_{3,1}\oplus \g_{1})^1$ and the associated topological quandle are respectively
\begin{footnotesize}	\[ M(w,\,x,\,y,\,z)\rhd_1 M(t,\,a,\,b,\,c)\begin{bmatrix} 1&x&ax+2\,bx+yb+w&0\\ \noalign{\medskip}0
	&1&y&0\\ \noalign{\medskip}0&0&1&0\\ \noalign{\medskip}0&0&0&{{\rm e}^
		{z}}\end{bmatrix} \] \end{footnotesize}
		\[Q(\mathit{G}_{3,1}\times \R)^1 =
	\{M(t,a,b,c)\in \mathit{G}_{3,1}\times \R|  \; a=-b \}.\]
	
	\item The Lie rack structure associated to $(\g_{3,1}\oplus \g_{1})^2$ and the associated topological quandle are respectively
\begin{footnotesize}\[M(w,\,x,\,y,\,z)\rhd_2 M(t,\,a,\,b,\,c)=\begin{bmatrix} 1&x&-ay+bx+yb+w&0\\ \noalign{\medskip}0&1
	&y&0\\ \noalign{\medskip}0&0&1&0\\ \noalign{\medskip}0&0&0&{{\rm e}^{z
	}}\end{bmatrix}.\]\end{footnotesize}
	\[Q(\mathit{G}_{3,1}\times \R)^2=\{M(t,a,b,c)\in \mathit{G}_{3,1}\times \R|  \; b=0 \}.\]
\item The Lie rack structure associated to $(\g_{3,1}\oplus \g_{1})^3$ and the associated topological quandle are respectively
\begin{footnotesize}\[M(w,\,x,\,y,\,z)\rhd_3 M(t,\,a,\,b,\,c) =\begin{bmatrix} 1&x&xa-ya+xb+\e yb+w&0\\ \noalign{\medskip}0
&1&y&0\\ \noalign{\medskip}0&0&1&0\\ \noalign{\medskip}0&0&0&{{\rm e}^
	{z}}\end{bmatrix} \] \end{footnotesize}
\[Q(\mathit{G}_{3,1}\times \R)^3 =\{M(t,a,b,c)\in \mathit{G}_{3,1}|  \; b=0 \}.\]
\item The Lie rack structure associated to $(\g_{3,1}\oplus \g_{1})^4$ and the associated topological quandle are respectively
\begin{footnotesize}	\[ M(w,\,x,\,y,\,z)\rhd_4 M(t,\,a,\,b,\,c)=\begin{bmatrix} 1&x&\gamma\,ay+\gamma\,bx-ya+xb+w&0
\\ \noalign{\medskip}0&1&y&0\\ \noalign{\medskip}0&0&1&0
\\ \noalign{\medskip}0&0&0&{{\rm e}^{z}}\end{bmatrix}\] \end{footnotesize}
\[Q(\mathit{G}_{3,1}\times \R)^4 =\{M(t,a,b,c)\in \mathit{G}_{3,1}\times \R|  \; ab=0 \}.\]	
\end{enumerate}
\item[$\bullet$ $\g_{3,2}\oplus \g_{1}$.] The associated simply-connected Lie group is given by 
$$\mathit{G}_{3,2}\times \R = \displaystyle\left\{  
\begin{bmatrix}
1 & 0 &0 & 0\\ 
x  & e^{y} &0 &0 \\
w &  -ye^{y}&e^y &0\\
0& 0&0 & e^z 
\end{bmatrix},
w,x,\,y, z \,\in \mathbb{R}\displaystyle\right\}.$$
The Lie rack structure associated to $(\g_{3,2}\oplus \g_{1})^1$ and the associated topological quandle are
\begin{footnotesize}
\begin{equation*}
M(w,\,x,\,y,\,z)\rhd M(t,\,a,\,b,\,c)=\\
\begin{bmatrix} 1&0&0&0\\ 
 (x-a){{\rm e}^{-b}}+a{{\rm e}^{-b+y}}&{{\rm e}^{y}}&0&0\\ 
 (bx-ab){{\rm e}^{-b}}+{{\rm e}^{-b+y}}(ab-ay) &-y{{\rm e}^{y}}&{{\rm e}^{y}}&0\\ 
 (w-t){{\rm e}^{-b}}t+t{{\rm e}^{-b+y}} & & & \\
	\noalign{\medskip}0&0&0&{{\rm e}^{yb+z}}\end{bmatrix} 
\end{equation*}\end{footnotesize}
\[Q(\mathit{G}_{3,2}\times \R)=\{M(t,a,b,c)\in \mathit{G}_{3,2}\times \R|  \; b=0 \}.\]
 \item[$\bullet$ $\g_{3,3}\oplus \g_{1}$.] The associated simply-connected Lie group is given by 
 $$\mathit{G}_{3,3}\times \R = \displaystyle\left\{  
 \begin{bmatrix}
 1 & 0 &0 & 0\\ 
 x  & e^{y} &0 &0 \\
 w &  0&e^y &0\\
 0& 0&0 & e^z 
 \end{bmatrix},
 w,x,\,y, z \,\in \mathbb{R}\displaystyle\right\}.$$
The Lie rack structure associated to $(\g_{3,3}\oplus \g_{1})^1$ and the associated topological quandle are respectively  
\begin{footnotesize}
\begin{equation*}
M(w,\,x,\,y,\,z)\rhd M(t,\,a,\,b,\,c)=
\begin{bmatrix} 1&0&0&0\\ \noalign{\medskip}-a{{\rm e}^{-
		b}}+{{\rm e}^{-b}}x+a{{\rm e}^{-b+y}}&{{\rm e}^{y}}&0&0
\\ \noalign{\medskip}-t{{\rm e}^{-b}}+{{\rm e}^{-b}}w+t{{\rm e}^{-b+y}
}&0&{{\rm e}^{y}}&0\\ \noalign{\medskip}0&0&0&{{\rm e}^{yb+z}}
\end{bmatrix} 
\end{equation*}
\end{footnotesize} 
\[Q(\mathit{G}_{3,3}\times \R)=\{M(t,a,b,c)\in \mathit{G}_{3,3}\times \R|  \; b=0 \}.\]
\item[$\bullet$ $\g^0_{3,4}\oplus \g_{1}$.] The associated simply-connected Lie group is given by 
$$\mathit{G}^0_{3,4}\times \R = \displaystyle\left\{  
\begin{bmatrix}
1 & 0 &0 & 0\\ 
w  &  \cosh(y) &-\sinh(y) &0 \\
x &  -\sinh(y)&\cosh(y) &0\\
0& 0&0 & e^z 
\end{bmatrix},
w,x,\,y, z \,\in \mathbb{R}\displaystyle\right\}.$$
The Lie rack structure associated to $(\g^0_{3,4}\oplus \g_{1})^1$ and the associated topological quandle are respectively 
\begin{footnotesize}\begin{equation*}
M(w,\,x,\,y,\,z)\rhd M(t,\,a,\,b,\,c)=
\begin{bmatrix}
1 & 0 &0 & 0\\ 
\sinh(y-b)a+\cosh(y-b)t+ &  \cosh(y) &-\sinh(y) &0 \\
 \sinh(b)(x-a)+\cosh(b)(w-t)  & & & \\
\sinh(y-b)t+\cosh(y-b)a+&  -\sinh(y)&\cosh(y) &0\\
 \sinh(b)(w-t)+\cosh(b)(x-a) & & & \\
0& 0&0 & e^{yb+z} 
\end{bmatrix}
\end{equation*}\end{footnotesize}
\[Q(\mathit{G}^0_{3,4}\times \R)=\{M(t,a,b,c)\in \mathit{G}^0_{3,4}\times \R|  \; b=0 \}.\]
\item[$\bullet$ $\g^\al_{3,4}\oplus \g_{1}$.] The associated simply-connected Lie group is given by 
$$\mathit{G}^\al_{3,4}\times \R = \displaystyle\left\{  
\begin{bmatrix}
1 & 0 &0 & 0\\ 
w  &  e^{\al y}\cosh(y) &-e^{\al y}\sinh(y) &0 \\
x &  -e^{\al y}\sinh(y)& e^{\al y}\cosh(y) &0\\
0& 0&0 & e^z 
\end{bmatrix},
w,x,\,y, z \,\in \mathbb{R}\displaystyle\right\}.$$
The Lie rack structure associated to $\g^\al_{3,4}\oplus \g_{1}$ and the associated topological quandle are respectively
\begin{footnotesize}
\begin{align*}
&M(w,\,x,\,y,\,z)\rhd M(t,\,a,\,b,\,c)=\\
&\begin{bmatrix}
1 & 0 &0 & 0\\ 
\sinh(y-b)ae^{\al(y-b)}+\cosh(y-b)te^{\al(y-b)}+&  e^{\al y} \cosh(y) &- e^{\al y}\sinh(y) &0 \\
\sinh(b)(x-a)e^{-\al b}+\cosh(b)(w-t)e^{-\al b} & & & \\
\sinh(y-b)te^{\al(y-b)}+\cosh(y-b)ae^{\al(y-b)}+&  - e^{\al y}\sinh(y)& e^{\al y}\cosh(y) &0\\
\sinh(b)(w-t)e^{-\al b}+\cosh(b)(x-a)e^{-\al b} & & & \\
0& 0&0 & e^{yb+z} 
\end{bmatrix}
\end{align*}
\end{footnotesize}
\[Q(\mathit{G}^\al_{3,4}\times \R)=\{M(t,a,b,c)\in \mathit{G}^\al_{3,4}\times \R|  \; b=0 \}.\]
\item[$\bullet$ $\g^0_{3,5}\oplus \g_{1}$.] The associated simply-connected Lie group is given by	
$$\mathit{G}^0_{3,5}\times \R = \displaystyle\left\{  
\begin{bmatrix}
1 & 0 &0 & 0 &0 \\ 
w  &  \cos(y) &-\sin(y) &0 &0 \\
x &  \sin(y)&\cos(y) &0&0\\
0& 0&0 & e^y &0\\
0& 0&0 & 0 &e^z\\
\end{bmatrix},
w,x,\,y, z \,\in \mathbb{R}\displaystyle\right\}.$$

The Lie rack structure associated to $\g^0_{3,5}\oplus \g_{1}$ and the associated topological quandle are respectively
\begin{footnotesize}\begin{equation*}
M(w,\,x,\,y,\,z)\rhd M(t,\,a,\,b,\,c)=\\
\begin{bmatrix}
1 & 0 &0 & 0 &0 \\ 
\sin(y-b)a+\cos(y-b)t+ &  \cos(y) &-\sinh(y) &0 &0  \\
\sin(b)(x-a)+\cos(b)(w-t)  &  & & \\
\sin(y-b)t+\cos(y-b)a+&  \sin(y)&\cos(y) &0 &0 \\ 
\sin(b)(t-w)+\cos(b)(x-a) & & & \\
0& 0&0 & e^y &0\\
0& 0&0 & 0& e^{yb+z} 
\end{bmatrix}
\end{equation*}\end{footnotesize}
\[Q(\mathit{G}^0_{3,5}\times \R)=\{M(t,a,b,c)\in \mathit{G}^0_{3,5}\times \R|  \; b=0 \}.\]
\item[$\bullet$ $\g^\al_{3,5}\oplus \g_{1}$.] The associated simply-connected Lie group is given by
 $$\mathit{G}^\al_{3,5}\times \R = \displaystyle\left\{  
\begin{bmatrix}
1 & 0 &0 & 0 &0 \\ 
w  &  e^{\al y}\cos(y) &-e^{\al y}\sin(y) &0 &0 \\
x &  e^{\al y}\sin(y)&e^{\al y}\cos(y) &0&0\\
0& 0&0 &  e^y &0\\
0& 0&0 & 0 &e^z\\
\end{bmatrix},
w,x,\,y, z \,\in \mathbb{R}\displaystyle\right\}.$$
The Lie rack structure associated to $\g^\al_{3,5}\oplus \g_{1}$ and the associated topological quandle are respectively
\begin{footnotesize}\begin{align*}
&M(w,\,x,\,y,\,z)\rhd M(t,\,a,\,b,\,c)=\\
&\begin{bmatrix}
1 & 0 &0 & 0 &0 \\ 
\sin(y-b)ae^{\al(y-b)}+\cos(y-b)te^{\al(y-b)}+ &   e^{\al y}\cos(y) &- e^{\al y}\sinh(y) &0 &0  \\
\sin(b)(x-a)e^{-\al b}+\cos(b)(w-t)e^{-\al b}  & & & \\
\sin(y-b)te^{\al(y-b)}+\cos(y-b)a+&   e^{\al y}\sin(y)& e^{\al y}\cos(y) &0 &0 \\ 
\sin(b)(t-w)e^{-\al b}+\cos(b)(x-a)e^{-\al b} & & & \\
0& 0&0 & e^y &0\\
0& 0&0 & 0& e^{yb+z} 
\end{bmatrix}
\end{align*}\end{footnotesize}

\[Q(\mathit{G}^\al_{3,5}\times \R)^3=\{M(t,a,b,c)\in \mathit{G}^\al_{3,5}\times \R|  \; b=0 \}.\]

\item[$\bullet$$\g_{4,1}$.] The associated simply-connected Lie group is given by 
$$\mathit{G}_{4,1} = \displaystyle\left\{  
\begin{bmatrix}
1 & z & \frac{1}{2}z^2& w\\ 
0 & 1 &z& w-x \\ 
0 & 0 & 1& y\\
0& 0& 0& 1
\end{bmatrix},	w,x,\,y,\,z,\,\in \mathbb{R}\displaystyle\right\}$$
\begin{enumerate}
	\item The Lie rack structure associated to $\g_{4,1}^1$ and the associated topological quandle are respectively
\begin{footnotesize}\begin{equation*}
	M(w,\,x,\,y,\,z)\rhd_1 M(t,\,a,\,b,\,c)=
 \begin{bmatrix}
	1 & z & \frac12 z^2& w+\frac12 c^2y+\frac12 bz^2+(yb+\e zc)(1+z)\\
	 & & &  -bcz+c(x-w)+z(t-a)\\ 
	0 & 1 &z&bz-cy+w-x +(yb+\e zc)\\ 
	0 & 0 & 1&y\\
	0& 0& 0&1 
	\end{bmatrix}
	\end{equation*}\end{footnotesize}
	\[Q(\mathit{G}_{4,1})^1=\{M(t,a,b,c)\in \mathit{G}_{4,1}|  \; b^2+\e c^2=0 , \e=0,1,-1\}.\]
	\item The Lie rack structure associated to $\g_{4,1}^2$ and the associated topological quandle are respectively
\begin{footnotesize}\begin{equation*}
	M(w,\,x,\,y,\,z)\rhd_2 M(t,\,a,\,b,\,c)=\\
\begin{bmatrix}
	1 & z & \frac12 z^2& w+\frac12 c^2y+\frac12 bz^2+(\e zc)(1+z)\\
	& & &  -bcz+c(x-w)+z(t-a)\\ 
	0 & 1 &z&bz-cy+w-x +\e zc\\ 
	0 & 0 & 1&y\\
	0& 0& 0&1 
	\end{bmatrix} 
	\end{equation*}\end{footnotesize}
	\[Q(\mathit{G}_{4,1})^2=\{M(t,a,b,c)\in \mathit{G}_{4,1}|  \;\e=0 \;\mbox{or } \; c=0 \}.\]
	\item The Lie rack structure associated to $\g_{4,1}^3$ and the associated topological quandle are respectively
\begin{footnotesize}\begin{equation*}
	M(w,\,x,\,y,\,z)\rhd_3 M(t,\,a,\,b,\,c)=\begin{bmatrix}
	1 & z & \frac12 z^2& w+\frac12 c^2y+\frac12 bz^2+(yc+ zb)(1+z)\\
	& & &  -bcz+c(x-w)+z(t-a)\\ 
	0 & 1 &z&bz-cy+w-x +(yc+ zb)\\ 
	0 & 0 & 1&y\\
	0& 0& 0&1 
	\end{bmatrix} 
	\end{equation*}\end{footnotesize}
	\[Q(\mathit{G}_{4,1})^3=\{M(t,a,b,c)\in \mathit{G}_{4,1}|  \;bc=0  \}.\]
\end{enumerate}
\item[$\bullet$$\g_{4,3}$.] The associated simply-connected Lie group is given by 
$$\mathit{G}_{4,3} = \displaystyle\left\{  
\begin{bmatrix}
e^{-z} & 0 & 0& w\\ 
0 & 1 &-z& x \\ 
0 & 0 & 1& y\\
0& 0& 0& 1
\end{bmatrix},\;
w,x,\,y,\,z,\,\in \mathbb{R}\displaystyle\right\}$$
\begin{enumerate}
	\item The Lie rack structure associated to $\g_{4,3}^1$ and the associated topological quandle	are respectively
\begin{footnotesize}\begin{equation*}
	M(w,\,x,\,y,\,z)\rhd_1 M(t,\,a,\,b,\,c)
	=\begin{bmatrix}
	e^{-z} & 0 & 0& te^{c-z}+e^{c}(w-t)\\ 
	0 & 1 &-z&cy-bz+x +zc\\ 
	0 & 0 & 1&y\\
	0& 0& 0&1 
	\end{bmatrix} 
	\end{equation*}\end{footnotesize}
	\[Q(\mathit{G}_{4,3})^1=\{M(t,a,b,c)\in \mathit{G}_{4,3}|  \;c=0  \}.\]
	\item The Lie rack structure associated to $\g_{4,3}^2$ and the associated topological quandle are respectively
\begin{footnotesize}\begin{equation*}
	M(w,\,x,\,y,\,z)\rhd_2 M(t,\,a,\,b,\,c)
	=\begin{bmatrix}
	e^{-z} & 0 & 0& te^{c-z}+e^{c}(w-t)\\ 
	0 & 1 &-z&cy-bz+x +(yb+ \e zc)\\ 
	0 & 0 & 1&y\\
	0& 0& 0&1 
	\end{bmatrix} 
	\end{equation*}\end{footnotesize}
	\[Q(\mathit{G}_{4,3})^2=\{M(t,a,b,c)\in \mathit{G}_{4,3}|  b^2+\e c^2=0 \}.\]
	\item The Lie rack structure associated to $\g_{4,3}^3$ and the associated topological quandle are respectively
	\begin{footnotesize}\begin{equation*}
	M(w,\,x,\,y,\,z)\rhd_3 M(t,\,a,\,b,\,c)
	=\begin{bmatrix}
	e^{-z} & 0 & 0& te^{c-z}+e^{c}(w-t)\\ 
	0 & 1 &-z&cy-bz+x +\e (yc+  zb)\\ 
	0 & 0 & 1&y\\
	0& 0& 0&1 
	\end{bmatrix} 
	\end{equation*}\end{footnotesize}
	\[Q(\mathit{G}_{4,3})^3=\{M(t,a,b,c)\in \mathit{G}_{4,3}|  bc=0  \}.\]
\end{enumerate}
\item[$\bullet$$\g^{-1}_{4,8}$.]  The associated simply-connected Lie group is given by 
	$$\mathit{G}^{-1}_{4,8} = \displaystyle\left\{  
\begin{bmatrix}
1& x & w\\ 
0 & e^{z} &y \\ 
0 & 0 & 1\\
\end{bmatrix},\;
w,x,\,y,\,z,\,\in \mathbb{R}\displaystyle\right\}.$$
The Lie rack structure associated to $(\g^{-1}_{4,8})^1$ and the associated topological quandle are respectively
\begin{footnotesize}\begin{align*}
&M(w,\,x,\,y,\,z)\rhd M(t,\,a,\,b,\,c)= \begin{bmatrix}
1 & a+xe^{c}-ae^{z} & w+ zc+ae^{-c}(b+y-be^{z}))\\ 
0 & e^{z} & (y-b)e^{-c}+be^{z-c}\\ 
0& 0& 1
\end{bmatrix}
\end{align*}\end{footnotesize}
\[Q(\mathit{G}^{-1}_{4,8})=\{M(t,a,b,c)\in \mathit{G}^{-1}_{4,8}|   c=0 \}.\]
\item[$\bullet$ $\g^{0}_{4,9}$.] The associated simply-connected Lie group is given by 
\begin{footnotesize}$$G^{0}_{4,9} = \displaystyle\left\{  
\begin{bmatrix}
1 & -x\cos(z)-y\sin(z)& y\cos(z)-x\sin(z)& -2w&0\\ 
0 & \cos(z) &\sin(z)& y& 0\\ 
0 & -\sin(z) & \cos(z) &x&0\\
0& 0& 0&1 &0\\
0& 0& 0&0 &e^{z}\\
\end{bmatrix},\;
w,x,\,y,\,z,\,\in \mathbb{R}\displaystyle\right\}.$$\end{footnotesize}
The Lie rack structure associated to $(\g^{0}_{4,9})^1$  and the associated topological quandle are respectively
\begin{footnotesize}\begin{align*}
&M(w,\,x,\,y,\,z)\rhd M(t,\,a,\,b,\,c)=\\
&\begin{bmatrix}
1 	&\cos(c+z)(a-x)+				&\cos(c+z)(y-b)+			& \cos(z)ay-\cos(z)bx+\sin(z)a^2+		&0\\ 
&\sin(c+z)(b-y)+				&\sin(c+z)(a-x)+			&-\sin(z)ax+\sin(z)b^2-\sin(z)by+			&\\
&-a\cos(c)-b\sin(c)			&b\cos(c)-a\sin(c)			&-2zc+ay-bx-2w						&\\
0 	&\cos(z) 							&\sin(z)							&\cos(c-z)b+ \sin(c-z)a+						&0\\ 
&										&									&\cos(c)(y-b)+\sin(c)(a-x)						&\\
0 	& -\sin(z) 							& \cos(z) 						&\cos(c-z)a- \sin(c-z)b+						&0\\
&										&									&\cos(c)(x-a)+\sin(c)(y-b)						&\\
0	&0									&0								&1 														&0\\
0	&0									&0								&0 														&e^{z}\\
\end{bmatrix}
\end{align*}\end{footnotesize}
\[Q(\mathit{G}^{0}_{4,9})=\{M(t,a,b,c)\in \mathit{G}^{0}_{4,9}|   c=0 \}.\]
\end{itemize}

\section{Some algebraic properties of the obtained topological quandles}\label{section5}

\subsection{Quasi-triviality}

We consider $(G,\rhd)$ a Lie rack associated to a symmetric Leibniz algebra. The rack operation is given by
\begin{equation}
{h \rhd g \, {\color{red}\xout{ : }}=\,g^{-1}hg \chi (h,\,g).}
\end{equation}
We consider the associated topological quandle
\[ Q(G)=\{ g\in G,\chi(g,g)=1  \}. \]
$Q(G)$ is quasi-trivial if and only if, for any $g,k\in Q(G)$ 
\[ g\rhd(g\rhd k)=g. \]
Let $g,k\in Q(G)$, We have
\begin{align*} g\rhd(g\rhd k)=g\rhd(k^{-1}gk\chi(g,k))&=\chi(g,k)^{-1}k^{-1}g^{-1}kgk^{-1}gk\chi(g,k)\chi(g,(k^{-1}gk\chi(g,k))\\&=k^{-1}g^{-1}kgk^{-1}gk\chi(g,(k^{-1}gk\chi(g,k))\\
&= k^{-1}g^{-1}kgk^{-1}gk,
\end{align*}
Since
\[ \chi(g_1g_2,h)=\chi(h,g_1g_2)=\chi(g_1,h)\chi(g_2,h)\ {{\rm and}}\ \chi(g,\chi(g,k))=\chi(g,g)=1. \]
Then $Q(G)$ is quasi-trivial if and only if
\[ [g,k^{-1}gk]=1,\ {{\rm for\ any}\ g,k\in Q(G).} \]where $[a,b]=aba^{-1}b^{-1}$.

Note that if the quandle $Conj(G)$ is quasi-trivial then $Q(G)$ is quasi-trivial.

We use this characterization to obtain the following result by a direct computation (see \cite{link}). 
\begin{prop}
	\begin{enumerate}
		\item 	The topological quandles $Q(G_{3,1})^1$,  $Q(G_{3,1})^2$, $Q(G_{3,1})^3$, and $Q(G_{3,1})^4$ are quasi-trivial.
		\item 	The topological quandles $Q(G_{2,1}\times \R^2)^2$, $Q(G_{2,1}\times \R^2)^3$ are quasi-trivial.
		\item   The topological quandles $Q(G_{3,1}\times \R)^1$, $Q(G_{3,1}\times \R)^2$, $Q(G_{3,1}\times \R)^3$, and $Q(G_{3,1}\times \R)^4$ are  quasi-trivial.
		\item   The topological quandles $Q(G_{3,2}\times \R)$, $Q(G_{3,3}\times \R)$, $Q(G^0_{3,4}\times \R)$, $Q(G^\al_{3,4}\times \R)$ , 
		$Q(G^0_{3,5}\times \R)$ and  $Q(G^\al_{3,5}\times \R)$ are quasi-trivial.
		\item  \begin{itemize}\item The topological quandle $Q(G_{4,1})^1$ is quasi-trivial if $\e= 0, 1$.
			\item  The topological quandle $Q(G_{4,1})^2$ is quasi-trivial if $\e \neq  0$.
			\item  The topological quandle $Q(G_{4,1})^3$ is quasi-trivial. 
		\end{itemize}
		\item \begin{itemize}
			\item  The topological quandle $Q(G_{4,3})^1$ is quasi-trivial.
			\item The topological quandle $Q(G_{4,3})^2$ is quasi-trivial if $\e=1$.\end{itemize}
		\item The topological quandles $Q(G^{-1}_{4,8})$ and $Q(G^0_{4,9})$ are quasi-trivial.
	\end{enumerate}
\end{prop}
It is important to point out that quasi-trivial quandles can be used to obtain link-homotopy invariants (see \cite{link}).

\subsection{Medial quandles} 
We prove that all classes of topological quandles $Q(G_{3,1})$ and $Q(G_{3,1}\times \R)$ are  medial.  {To do that we need} the following lemma {which can be easily shown}.
\begin{lem}
	\begin{enumerate}
	\item For any $M(a,b,c), M(x,y,z),\; M(t,u,w) \in G_{3,1}$, we have
	\begin{eqnarray*}
		M(a,b,c)\rhd[M(x,y,z)\rhd M(t,u,w)]&=& M(a,b,c)\rhd M(x,y,z),\\
		M(a,b,c)\rhd^{-1}[M(x,y,z)\rhd^{-1} M(t,u,w)]&=& M(a,b,c)\rhd^{-1} M(x,y,z),\\
		M(a,b,c)\rhd[M(x,y,z)\rhd^{-1} M(t,u,w)]&=& M(a,b,c)\rhd M(x,y,z).
	\end{eqnarray*}

  \item  For any $M(t, a,b,c), M(w, x,y,z), M(s,r,u,v) \in (G_{3,1}\times \R)$, we have
  \begin{eqnarray*}
  	M(t, a,b,c)\rhd[M(w, x,y,z)\rhd M(s,r,u,v)]&=&	M(t, a,b,c)\rhd M(w, x,y,z),\\
    M(t, a,b,c)\rhd^{-1}[M(w, x,y,z)\rhd^{-1}  M(s,r,u,v)]&=&   M(t, a,b,c)\rhd^{-1} M(w, x,y,z),\\
    M(t, a,b,c)\rhd[M(w, x,y,z)\rhd^{-1}M(s,r,u,v)]&=& M(t, a,b,c)\rhd  M(t, a,b,c).
     \end{eqnarray*}

\end{enumerate}
\end{lem}
\begin{proof}
	By straightforward computation. 
\end{proof}
Thus, we get 
\begin{prop}
{Both} classes of quandles $Q(G_{3,1})$ and $Q(G_{3,1}\times \R)$ are medial.
\end{prop} 
\begin{proof}
	Let $(M( a_i,b_i,c_i))_{1\leq i\leq 4} \in (G_{3,1})^4$. Due to the above lemma, we have
	\begin{eqnarray*}
		(M( a_1,b_1,c_1)&\rhd& M( a_2,b_2,c_2))\rhd (M(a_3,b_3,c_3)\rhd  M(a_4,b_4,c_4))\\
		&=&(M( a_1,b_1,c_1)\rhd M( a_2,b_2,c_2))\rhd M( a_3,b_3,c_3),\\
		&=&(M( a_1,b_1,c_1)\rhd M( a_3,b_3,c_3))\rhd (M( a_2,b_2,c_2)\rhd M( a_3,b_3,c_3)), \; \mbox{(self-distributivity)}\\
		&=& (M( a_1,b_1,c_1)\rhd M( a_3,b_3,c_3))\rhd M( a_2,b_2,c_2),\\
		&=& (M( a_1,b_1,c_1)\rhd M( a_3,b_3,c_3))\rhd (M( a_2,b_2,c_2)\rhd  M(a_4,b_4,c_4)).
	\end{eqnarray*}
Hence, the quandles $Q(G_{3,1})$ are medial. Similarly, we show that $Q(G_{3,1}\times \R)$ are medial. 
\end{proof}

\section{Example of computation}\label{section6}
In this section we apply Method~\ref{1} to construct the Lie racks associated to symmetric Leibniz algebras with underlying Lie algebra $\g_{4,1}$. 
To start, we determine the symmetric Leibniz algebras associated to $\g_{4,1}$. For that, we proceed as described in~\ref{om}.
Consider the Lie algebra $\g_{4,1}$ with non-vanishing Lie brackets given by 
  $$[e_2,\, e_4]= e_1,\;[e_3,\, e_4]= e_2.$$
The center is $Z(\g_{4,1})= \R e_1$. By applying Proposition~\ref{3.1} and by doing a straightforward computations, we get the corresponding symmetric bilinear $\om$ satisfying the equation~\eqref{eq}.  
 \begin{eqnarray*}
 	\om(e_3,e_3)&=&  \al e_{1},\\
 	\om(e_4,e_4)&=&  \be e_{1},\\
 	\om(e_3,e_4)&=&\ga e_1
 \end{eqnarray*}
where $(\alpha,\, \beta,\, \gamma) \neq 0$. 

We use now Proposition \ref{equi} to complete the classification.
We consider the group of the automorphisms of $\g_{4,1}$ given by 
$$T=\left[ \begin {array}{cccc} a_{{3,3}}{a_{{4,4}}}^{2}&a_{{2,3}}a_{{4,4
}}&a_{{1,3}}&a_{{1,4}}\\ \noalign{\medskip}0&a_{{3,3}}a_{{4,4}}&a_{{2,
		3}}&a_{{2,4}}\\ \noalign{\medskip}0&0&a_{{3,3}}&a_{{3,4}}
\\ \noalign{\medskip}0&0&0&a_{{4,4}}\end {array} \right] $$
 
with $\det(T)= a_{3,3}^2a^4_{4,,4}$. We consider $\mu(u,v)=T^{-1}\om(Tu,Tv)$. Then 
\begin{eqnarray*}
	\mu(e_3,e_3)&=&  \frac{a_{3,3}}{a^2_{4,4}}\al e_{1},\\
	\mu(e_4,e_4)&=&  \frac{\al a^2_{3,4}+\be a^2_{4,4}+ 2\ga a_{3,4} a_{4,4}}{a_{3,3}a^2_{4,4}} e_{1},\\
	\mu(e_3,e_4)&=& \frac{\al a_{3,4}+\ga a_{4,4}}{a^2_{4,4}} e_1.
\end{eqnarray*}
Thus, we have three cases: 
\begin{enumerate}
	\item If $\al \neq 0$, we take $a_{3,3} = \frac{a^2_{4, 4}}{\al}$ and $a_{3,4} = \frac{-\ga a_{4,4}}{\al}$. Hence, we have 
	  \begin{eqnarray*}
	  \mu(e_3,e_3)&=&  e_{1},\\
	  \mu(e_4,e_4)&=&- \frac{\al \be -\ga^2}{a^2_{4,4}} e_{1}.
      \end{eqnarray*}
   \item If $\al = 0$ and $\ga =0$, we have $	\mu(e_4,e_4)= \frac{\be}{a_{3,3}}$.
     \item If $\al=0$ and $\ga \neq 0$, we take $a_{4, 4} = \ga, a_{3, 4} = -\frac{\be}{2}$. Then we get 
     $\mu (e_3,e_4)= e_1$.
    \end{enumerate}
Therefore, we obtain three classes of symmetric Leibniz algebras  {whose} underlying Lie algebra {is} $\g_{4,1}$ given in Table~\ref{tab1}. 

We finish this subsection by  {applying} explicitly the method~\ref{me} to get the Lie rack structure on the Lie group $\mathit{G}_{4,1}$ associated to the symmetric Leibniz algebra $\g^1_{4,1}$. We have $\g^1_{4,1}.\g^1_{4,1}\simeq \langle e_1,e_2\rangle$ and the quotient space $\mathfrak{a}:= \g_{4,1}/ \g^1_{4,1}.\g^1_{4,1}$ is identified to $\R^2$. So the projection $q_{1}\, : \g_{4,1} \longrightarrow \mathfrak{a}$ is given by 
$$ q_{1}(w,x, \,y,\,z)= (y,\,z),$$
it follows that the homomorphism $\kappa \,: \mathit{G}_{4,1}\longrightarrow \mathbb{R}^{2}$ must be defined by $ \kappa(w,x,\,y,\,z)\,= (y,\,z)$.\\
Now, $\beta_1 \,: \mathbb{R}^{2}  \times \mathbb{R}^{2} \too  \g_{4,1}$ is defined by 
	$$\be_{1}((y,\,z), (b,\,c))= (yb+\e zc)e_1, \quad \e=0,1,-1 $$
and the map $ \chi_1 \, :\mathit{G}_{4,1}\times \mathit{G}_{4,1} \longrightarrow \mathit{G}_{4,1}$ is given by 
	\begin{equation*}
	\chi_1 ((w,x,\,y,\,z),\,(t, a,\,b,\,c))\,=\begin{bmatrix}
	1 & 0 & 0& (yb+\e zc) \\ 
	0 & 1 &0& (yb+\e zc)\\ 
	0 & 0 & 1&0\\
	0& 0& 0&1 
	\end{bmatrix}. 
	\end{equation*}
	Finally, we get the Lie rack product on $\mathit{G}_{4,1}$
	\begin{footnotesize}\begin{equation*}
		M(w,\,x,\,y,\,z)\rhd_1 M(t,\,a,\,b,\,c)=
		\begin{bmatrix}
		1 & z & \frac12 z^2& w+\frac12 c^2y+\frac12 bz^2+(yb+\e zc)(1+z)\\
		& & &  -bcz+c(x-w)+z(t-a)\\ 
		0 & 1 &z&bz-cy+w-x +(yb+\e zc)\\ 
		0 & 0 & 1&y\\
		0& 0& 0&1 
		\end{bmatrix}
		\end{equation*}
	\end{footnotesize}

\newpage

\begin{center}
	\begin{footnotesize}
		\begin{tabular}{|c|c|c|c|c|}
			\hline
			Lie algebra&Non-vanishing &Symmetric Leibniz&Name&Conditions\\
			&Lie brackets&Non-vanishing brackets&&\\
			\hline
			\multirow{4}{*}{$\mathfrak{g}_{3,1}$}& \multirow{4}{*}{$[e_2,e_3]=e_1$}&  $e_2.e_2 = e_1,\;$ $e_3.e_3=e_1,\;$$e_2.e_3 =2e_1$&$\mathfrak{g}_{3,1}^1$&\\
			\cline{3-4}
			&&$e_3.e_3 = e_1,\;$$e_2.e_3 =-e_3.e_2=e_1$&$\mathfrak{g}_{3,1}^2$&\\
			\cline{3-4}
			&&$e_2.e_2 = e_1,\;$ $e_3.e_3=\e e_1,\;$$e_2.e_3 = -e_3.e_2=e_1$& 
			$\mathfrak{g}_{3,1}^3$&$\e =0,1$\\
			\cline{3-4}
			&&$e_2.e_3 = (\ga +1)e_1,\;$$e_3.e_2 = (\gamma -1)e_1$&$\mathfrak{g}_{3,1}^4$&$\ga \neq 0$\\
			\hline
			
			\multirow{5}{*}{ }& \multirow{5}{*}{}&  $e_2.e_3 =e_3.e_2= \e (e_3+e_4),\;$ && \\
			& & $e_2.e_4 =e_4.e_2= -\e (e_3+e_4),$&  $(\mathfrak{g}_{2,1}\oplus 2\mathfrak{g}_{1})^{1}$ & $\e \neq 0$\\
			& &  $e_1.e_2 =-e_2.e_1= e_1$ & &\\
			\cline{3-4}
			& &$e_2.e_2 =e_3+e_4,\;$ $e_2.e_3 =e_3.e_2= \e (e_3+e_4),$ &  & \\
			$\mathfrak{g}_{2,1}\oplus 2\mathfrak{g}_{1}$&$[e_1,\, e_2]= e_1$ & $e_2.e_4 =e_4.e_2= -\e (e_3+e_4),$ & $(\mathfrak{g}_{2,1}\oplus 2\mathfrak{g}_{1})^{2}$ & $\e \neq 0$\\
			&&  $e_1.e_2 =-e_2.e_1= e_1$ & & \\
			\cline{3-4}
			&& $e_2.e_2 =  e_3,\;$  $e_1.e_2 =-e_2.e_1= e_1$& 
			$(\mathfrak{g}_{2,1}\oplus 2\mathfrak{g}_{1})^{3}$& \\
			\cline{3-4}
			&&$e_2.e_4 =e_4.e_2= e_3,\;$$e_1.e_2 =-e_2.e_1= e_1$ &	$(\mathfrak{g}_{2,1}\oplus 2\mathfrak{g}_{1})^{4}$&\\
			\cline{3-4}
			&& $e_2.e_2=\ga e_3,\;$ $e_3.e_3 = \e e_3$ & 	$(\mathfrak{g}_{2,1}\oplus 2\mathfrak{g}_{1})^{5}$&$\ga=0,1$\\
			&& $e_1.e_2 =-e_2.e_1= e_1$ &&$\e=0,1,-1$\\
			\hline

			\multirow{4}{*}{$\mathfrak{g}_{3,1}\oplus \mathfrak{g}_{1}$}& \multirow{4}{*}{$[e_2,e_3]=e_1$}&  $e_2.e_2 = e_1,\;$ $e_3.e_3=e_1,\;$$e_2.e_3 =2e_1$&$(\mathfrak{g}_{3,1}\oplus \mathfrak{g}_{1})^1$&\\
			\cline{3-4}
			&&$e_3.e_3 = e_1,\;$$e_2.e_3 =-e_3.e_2=e_1$&$(\mathfrak{g}_{3,1}\oplus \mathfrak{g}_{1})^2$&\\
			\cline{3-4}
			&&$e_2.e_2 = e_1,\;$ $e_3.e_3=\e e_1,\;$$e_2.e_3 = -e_3.e_2=e_1$& 
			$(\mathfrak{g}_{3,1}\oplus \mathfrak{g}_{1})^3$&$\e =0,1$\\
			\cline{3-4}
			&&$e_2.e_3 = (\ga +1)e_1,\;$$e_3.e_2 = (\gamma -1)e_1$&$(\mathfrak{g}_{3,1}\oplus \mathfrak{g}_{1})^4$&$\ga \neq 0$\\
			
			\hline 	$\mathfrak{g}_{3,2}\oplus \mathfrak{g}_{1}$   &  $[e_3,\, e_1]= e_1$   &$e_{3}.e_{3} =  e_{4},\; $ $e_{3}.e_{1} = -e_{3}.e_{1}= e_1$  &     $(\mathfrak{g}_{3,2}\oplus \mathfrak{g}_{1})^1$ &  \\
			&  $[e_2,\, e_3]= e_{1}- e_2$  & $e_{2}.e_{3} = -e_{3}.e_{2} = e_{1}-e_2$  & &  \\

			\hline $\mathfrak{g}_{3,3}\oplus \mathfrak{g}_{1}$ & $[e_3,\, e_1]= e_1$  & $e_{3}.e_{3} =   e_{4},\;$ $e_{3}.e_{1} = -e_{1}.e_{3}= e_{1}$  &$(\mathfrak{g}_{3,3}\oplus \mathfrak{g}_{1})^1$  &  \\
			& $[e_2,\, e_3]= - e_2$ & $e_{2}.e_{3} = -e_{3}.e_{2} = -e_2$ &  &  \\
			
			\hline  $\mathfrak{g}^{0}_{3,4}\oplus \mathfrak{g}_{1}$  &  $[e_2,\, e_3]=  e_1$& $e_{3}.e_{3} =  e_{4}, \;$ $e_{2}.e_{3} = -e_{3}.e_{2}= e_{1}$  &$(\mathfrak{g}^{0}_{3,4}\oplus \mathfrak{g}_{1})^1$   &  \\
			& $[e_3,\, e_1]= -e_2$  &  $e_{3}.e_{1} = -e_{1}.e_{3}= -e_{2}$ &   & \\
			
			\hline  $\mathfrak{g}^{\al}_{3,4}\oplus \mathfrak{g}_{1}$ & $[e_2,\, e_3]=  e_1-\al e_2$  & $e_{3}.e_{3} =  e_{4},\;$  $e_{2}.e_{3} = -e_{3}.e_{2}= e_1-\al e_2$& $(\mathfrak{g}^{\al}_{3,4}\oplus \mathfrak{g}_{1})^1$   & $\alpha >  0$ \\
			& $[e_3,\, e_1]= \al e_1- e_2$  & $e_{3}.e_{1} = -e_{1}.e_{3}= \al e_1- e_2$  & & $\al \neq 1$ \\

			\hline $\mathfrak{g}^{0}_{3,5}\oplus \mathfrak{g}_{1}$  & $[e_2,\, e_3]=  e_1$  & $e_{3}.e_{3} =   e_{4},\;$$e_{2}.e_{3} = -e_{3}.e_{2}= e_1$ &     $(\mathfrak{g}^{0}_{3,5}\oplus \mathfrak{g}_{1})^1$ & \\
			&$[e_3,\, e_1]= e_2$ &$e_{3}.e_{1} = -e_{1}.e_{3}=  e_2$  &    & \\
			
			\hline  $\mathfrak{g}^{\al}_{3,5}\oplus \mathfrak{g}_{1}$ & $[e_2,\, e_3]=  e_1-\al e_2$ &  $e_{3}.e_{3} =   e_{4},\;$$e_{2}.e_{3} = -e_{3}.e_{2}= e_1-\al e_2$  & $(\mathfrak{g}^{\al}_{3,5}\oplus \mathfrak{g}_{1})^1$& \\
			&$[e_3,\, e_1]= \al e_1+ e_2$ &$e_{3}.e_{1} = -e_{1}.e_{3}= \al e_1+ e_2$   &  &  $\alpha >  0$ \\
			\hline
			\multirow{3}{*}{$\mathfrak{g}_{4,1}$ }&\multirow{3}{*}{$[e_2,\, e_4]= e_1$}& $e_{3}.e_{3} =  e_{1},\: e_{4}.e_{4} =  \e e_{1} $& $\mathfrak{g}^1_{4,1}$ & $\e=0,1,-1$ \\
			& & $e_{3}.e_{4} = -e_{4}.e_{3}=  e_{2},\;$$e_{2}.e_{4} =-e_{4}.e_{2} = e_{1} $  &  &\\
			\cline{3-4}
			&$[e_3,\, e_4]= e_2$&$ e_{4}.e_{4} =  \e e_{1} $$e_{3}.e_{4} = -e_{4}.e_{3} = e_{2}$&$\mathfrak{g}_{4,1}^2$& $\e= 0,1$\\
			& & $e_{2}.e_{4} =-e_{4}.e_{2} = e_{1} $& & \\
			\cline{3-4}
			& &$e_{3}.e_{4} = e_1+e_2,\;$  $e_{4}.e_{3} = e_1- e_{2},\;$$e_{2}.e_{4} =-e_{4}.e_{2} = e_{1} $ & 
			$\mathfrak{g}_{4,1}^3$&\\
			\hline
			\multirow{3}{*}{$\mathfrak{g}_{4,3}$ }&\multirow{3}{*}{ $[e_1,\, e_4]= e_1$ }& $e_{4}.e_{4} =   e_{2},\;$$e_{1}.e_{4} = -e_{4}.e_{1} = e_{1}$& $\mathfrak{g}^1_{4,3}$ & \\
			& &  $e_{3}.e_{4} =-e_{4}.e_{3}= e_{2}$  &  &\\
			\cline{3-4}
			&$[e_3,\, e_4]= e_2$&$e_{3}.e_{3}=e_2,\; e_{4}.e_{4} =   \e e_{2} $ &$\mathfrak{g}_{4,3}^2$& $\e \in \R$\\
			& &  $e_{1}.e_{4} = -e_{4}.e_{1} = e_{1},\;$ $e_{3}.e_{4} =-e_{4}.e_{3}= e_{2}$& & \\
			\cline{3-4}
			& & $e_{1}.e_{4} = -e_{4}.e_{1} = e_{1}, \;$$e_{3}.e_{4} = (\e +1) e_{2}$& 
			$\mathfrak{g}_{4,3}^3$&  $\e \neq 0$\\
			&  &$e_{4}.e_{3} = (\e -1) e_{2}$ &  & \\
			
			\hline  $\mathfrak{g}^{-1}_{4,8}$ & $[e_2,\, e_3]= e_1$  &$e_{4}.e_{4} =  e_{1},\; $$e_{2}.e_{3} = -e_{3}.e_{2}= e_{1}  $&  &  \\
			& $[e_3,\, e_4]= -e_3$  & $e_{3}.e_{4} = -e_{4}.e_{3}= -e_{3}$ &  $(\mathfrak{g}^{-1}_{4,8})^1$   & \\
			& $[e_2,\, e_4]= e_2$ & $e_{2}.e_{4} = -e_{4}.e_{2}= e_{2},\;$  &   &  \\

			\hline $\mathfrak{g}^{0}_{4,9}$ &$[e_2,\, e_3]= e_1$  &  $e_{4}.e_{4} =   e_{1},\;$ $e_{2}.e_{3} = -e_{3}.e_{2}= e_{1}$&   &  \\
			&  $[e_2,\, e_4]=- e_3$ &  $e_{2}.e_{4} = -e_{4}.e_{2}=- e_{3}$  &$(\mathfrak{g}^{0}_{4,9})^1$& \\
			&$[e_3,\, e_4]= e_2$ & $e_{3}.e_{4} = -e_{4}.e_{3}= e_{2}$ &  & \\
			
			\hline  
			
		\end{tabular}\captionof{table}{\mbox{Symmetric Leibniz algebras of dimension 3 and 4.} \label{tab1} }	
	\end{footnotesize}\end{center}


\begin{thebibliography}{00}
	
	
	
	\bibitem{ayp}
	S. Albeverio, Sh.A. Ayupov, B.A. Omirov, {\it On nilpotent and simple Leibniz algebras}, Comm. in Algebra, {\bf 33(1)} (2005) pp. 159--172.
	
\bibitem{link} Ayumu Inoue,	{\it Quasi triviality for quandles for link-homotopy}, Journal of Knot Theory and Its Ramifications {\bf Vol. 22}, No. {\bf 06}, 1350026 (2013).
	
	\bibitem{ayp1}
	Sh.A. Ayupov, B.A. Omirov, {\it On Leibniz algebras}, in Algebra and operator theory (Tashkent, 1997), Kluwer Acad. Publ., Dordrecht, (1998) pp. 1-12.
	\bibitem{ayp2}
	Sh.A. Ayupov, B.A. Omirov, {\it On some classes of nilpotent Leibniz algebras}, Siberian Math. J.,  {\bf 42} (2001) pp.15-24.
	
	\bibitem{saidbar} 
	E. Barreiro and S. Benayadi, {\it A New Approach to Leibniz Bialgebras },
	Algebra and Representation Theory (2016) {\bf 19}:71-101.
	
	\bibitem{benbor} S. Benayadi and Matrin Bordemann, Private communication.
	\bibitem{boucetta}
	S. Benayadi, M. Boucetta, { \it Special bi-invariant linear connections on Lie
		groups and finite dimensional Poisson structures}, Differential Geometry and Applications 
	{ \bf 36} (2014) 66-89. 
	\bibitem{hidri}
	S. Benayadi, S. Hidri, {\it Quadratic Leibniz algebras}, Journal of Lie Theory
	{\bf 24}(2014) 737-759.
	\bibitem{biggs}
	R. Biggs,and C.C . Remsing, {\it On the classification of real four-dimensional Lie groups}, J. Lie Theory,
	{\bf 26} (2016) 4, 1001-1035
	
	
	\bibitem{bloh1}
	A.M. Bloh, { \it Cartan-Eilenberg homology theory for a generalized class of Lie algebras}, Dokl. Akad. Nauk SSSR 175 (1967), 824-826.		
	
	\bibitem{bloh}
	A.M. Bloh, {\it On a generalization of the concept of Lie algebra}, Dokl. Akad. Nauk SSSR 165
	(1965), 471-473.
	\bibitem{bord}
	M. Bordemann, and  F. Wagemann, {\it 
		Global Integration of Leibniz Algebras}, Journal of Lie Theory {\bf 27} (2017), No. 2, 555567.
	
	\bibitem{bri}
	E.Brieskorn, { \it Automorphic sets and braids and singularities}, Contemporary math. {\bf 78}(1988) 45-115.
	
	
	\bibitem{carter1}
	J.S. Carter, D. Jelsovsky, S. Kamada, L. Langford, M. Saito, {\it Computations of quandle cocycle invariants of knotted curves and surfaces}, Advances in math., {\bf 157 } (2001) 36-94.
	
	\bibitem{carter2} J. Scott Carter, Alissa S. Crans, Mohamed Elhamdadi, and Masahico Saito, {\it Cohomology of categorical
	self-distributivity}, J. Homotopy Relat. Struct. {\bf 3} (2008), no. 1, 13-63. 
	
	
	\bibitem{carter}
	J.S. Carter, D. Jelsovsky, S. Kamada, L. Langford, M. Saito, {\it Quandle
		cohomology and state-sum invariants of knotted curves and surfaces}, preprint at
	http://xxx.lanl.gov/abs/math.GT/9903135.
	
	\bibitem{covez}
	S. Covez, {\it The local integration of Leibniz algebras},  Annales de l'institut Fourier {\bf 63} (2013) 1-35.
	\bibitem{elhamdadi}
	M. Elhamdadi, and S. Nelson, {\it Quandles},  American Mathematical Soc. {\bf 74} (2015).
	
	\bibitem{fenn}
	R. Fenn, and C. Rourke, {\it Racks and links in codimension two},  Journal of Knot theory and its Ramifications {\bf 1} (1992) 343-406.
	\bibitem{mason}
	M. Geoffrey, Y. Gaywalee,  {\it Leibniz algebras and Lie algebras},  SIGMA Symmetry
	Integrability Geom. Methods Appl. {\bf 9}(2013).
	
	\bibitem{hidrith}
	S. Hidri, {\it Formes bilin{\'e}aires invariantes sur les alg{\`e}bres de Leibniz et les syst{\`e}mes triples de Lie (resp. Jordan) }, Universit{\'e} de Lorraine (2016).
	
	\bibitem{hermann}
	J. Hilgert and K.-Hermann Neeb, {\it Structure and geometry of Lie groups}, Springer Science \& Business Media (2011).
	\bibitem{joyce}
	D. Joyce, {\it A classifying invariant of knots, the knot quandle}, Journal of Pure and Applied      Algebra {\bf 23} (1982) 37-65.
	 
	\bibitem{inoue}
	 A. Inoue, {\it Quasi-triviality of quandles for link-homotopy}, Journal of Knot Theory and Its Ramifications {\bf 22} (2013).
	 
	\bibitem{medial} 
	P.Jedlička, A. Pilitowska, D. Stanovský and A. Zamojska-Dzienio, {\it The structure of medial quandles}, {\bf 443} (2015) 300 - 334.
	\bibitem{kauff}
	L.H. Kauffman, {\it Knots and Physics}, World Scientific (1991).
	
	\bibitem{kinyon}
	M. Kinyon,  {\it Leibniz algebras, Lie racks, and digroups}, Journal of Lie Theory
	Volume {\bf 17} (2007) 99–114.
	\bibitem{loday}
	J.-L. Loday, {\it Une version non-commutative des algebres de Lie}, L'Ens. Math {\bf 39} (1993) 269-293.
	
	\bibitem{matveev}
	S. V. Matveev, {\it Distributive groupoids in knot theory},  Sbornik: Mathematics {\bf 47}
	(1984) 73-83.
	\bibitem{davidl}
	D. Stanovský,  
    {\it A guide to self-distributive quasigroups, or latin quandles}, Quasigroups and Related Systems {\bf 23} (2015), 91-128.
	%
	\bibitem{rub}
	R.L. Rubinsztein, {\it Topological quandles and invariants of links}, Journal of Knot Theory and Its Ramifications
	{\bf 16} (2007) 789-808.
	
	
	
\end{thebibliography}
\end{document}